\documentclass[11pt]{article}

\usepackage{amsmath,amssymb,amscd,amsthm,amsfonts}
\usepackage{graphicx,subfigure}
\usepackage{hyperref}
\usepackage{dsfont}
\usepackage{color}

\newtheorem{theorem}{Theorem}[section]
\newtheorem*{theoremp}{Theorem}
\newtheorem{lemma}[theorem]{Lemma}
\newtheorem{claim}[theorem]{Claim}
\newtheorem{corollary}[theorem]{Corollary}
\newtheorem{conjecture}[theorem]{Conjecture}

\newtheorem{question}[theorem]{Question}

\newcommand{\rr}{\mathds{R}}
\newcommand{\C}{\mathcal{C}}
\newcommand{\G}{\mathcal{G}}

\newcommand{\ff}{\mathcal{F}}

\newcommand{\ttt}{{\mathcal T}}

\def\rr{\mathds{R}}
\DeclareMathOperator{\conv}{conv}
\DeclareMathOperator{\diam}{diam}

\DeclareMathOperator{\vol}{vol}

\title{Helly-type theorems for the diameter}

\author{Pablo Sober\'on}

\begin{document}

\maketitle

\begin{abstract}
	We study versions of Helly's theorem that guarantee that the intersection of a family of convex sets in $R^d$ has a large diameter.  This includes colourful, fractional and $(p,q)$ versions of Helly's theorem.  In particular, the fractional and $(p,q)$ versions work with conditions where the corresponding Helly theorem does not.  We also include variants of Tverberg's theorem, B\'ar\'any's point selection theorem and the existence of weak epsilon-nets for convex sets with diameter estimates.
\end{abstract}

\section{Introduction}

Quantitative results in combinatorial geometry have recently caught new interest.  Those surrounding Helly's theorem have as aim to show that \textit{given a finite family of convex set in $\rr^d$, if the intersection of every small subfamily is large, then the intersection of the whole family is also large} \cite{Amenta:2015tp}.  When the size of a convex set is measured according to a function that varies discretely, such as the number of points in a lattice, there are very sharp results (e.g. \cite{Aliev:2014va}).  However, when the size of a convex set is measured according to a function that varies continuously, such as the volume or diameter, the behaviour changes considerably.

The first results of this kind were presented by B\'ar\'any, Katchalski and Pach \cite{Barany:1982ga, Barany:1984ed}, where Helly-type theorems were made regarding the volume and diameter of the intersection of families of convex sets.  They showed that, given a finite family of convex sets in $\rr^d$, if the intersection of every $2d$ of them has volume at least one, one can obtain lower bounds on the volume of the intersection, and the same holds for the diameter.  The constant $2d$ is optimal, but the downside is that the guarantee of the diameter or volume of the intersection decreases quickly with the dimension.

Helly's theorem has an impressive number of variations and generalisations (see, for instance, the surveys \cite{Danzer:1963ug,Eckhoff:1993uy, Matousek:2002td, Wenger:2004uf, Amenta:2015tp}).  Thus, it is natural to determine which results can be extended in this quantitative framework.  For the volume, several advances have been made in this direction \cite{Naszodi:2015vi, DeLoera:2015wp, Soberon:2015tsa}.  This includes optimising the original result by B\'ar\'any, Katchalski and Pach, and finding colourful versions, fractional versions and $(p,q)$ type theorems.  These are classical variations of Helly's theorem found in \cite{Barany:1982va}, \cite{Katchalski:1979bq} and \cite{Alon:1992gb}, respectively.  The aim of this paper is to present analogues to these results for the diameter.  

For example, the guarantee of the size of the intersection can be improved if we are willing to check larger families.  Regarding the diameter, the following result makes this clear.

\begin{theoremp}[Helly's theorem for diameter, De Loera et al. {\cite[Thm 1.5]{DeLoera:2015wp}}]
Let $d$ be a positive integer and $1>\delta>0$, then, there is an integer $n=n(\diam, d, \delta)$ such that for any finite family 	$\ff$ of convex sets in $\rr^d$, if the intersection of every subfamily of size $n$ has diameter greater than or equal to one, then $\diam (\cap \ff) \ge 1- \delta$.  Moreover, $n(\diam, d, \delta) = \Omega_d (\delta^{-(d-1)/2})$.
\end{theoremp}

Given two functions $g(d,\delta)$ and $f(d,\delta)$, we say $g(d,\delta)=\Omega_d (f(\delta))$ if, for any fixed $d$, $g(d,\delta) = \Omega (f(\delta))$, and similarly with other notation for asymptotic bounds.  For an upper bound to the result above, one can apply the main theorem of \cite{Langberg:2009go} to get $n(\diam, d, \delta) = O(\delta^{-d/2})$.  The equivalent result for volume {\cite[Thm 1.4]{DeLoera:2015wp}} has similar upper and lower bounds in terms of $\delta$, giving $n(\vol, d, \delta)= \Theta_d (\delta^{-(d-1)/2})$.  In the same spirit as Lov\'asz's generalisation of Helly's theorem \cite{Barany:1982va}, we show a ``colourful version'' of De Loera et al.'s diameter Helly in Section \ref{section-colourful}.  Moreover, Theorem \ref{theorem-colourful-diameter} implies an asymptotically optimal bound for the diameter as well: $n(\diam, d, \delta) = \Theta_d (\delta^{-(d-1)/2})$.  Asymptotic result as above hold for a very general family of functions, and are closely related to the approximability of convex sets by polyhedra \cite{Amenta:2015tp}.

If we allow for a loss of diameter $\delta$, other versions of Helly's theorem can be recreated.  In particular, we show a version of Alon and Kleitman's $(p,q)$ theorem \cite{Alon:1992gb} for the diameter.  The $(p,q)$ theorem was conjectured originally by Hadwiger and Gr\"unbaum \cite{Hadwiger:1957we}, and asks if a slight weakening of Helly's condition is still enough information to bound the number of points needed to intersect all members of a finite family of convex sets in $\rr^d$.  A lucid description of the theorem and its variations is contained in a survey by Eckhoff \cite{Eckhoff:2003ed}, and more recent results are summarised in \cite{Amenta:2015tp}.

\begin{theorem}[$(p,q)$ theorem for diameter]\label{theorem-p,q-diameter}
Let $p \ge q \ge 2d$ be positive integers and $1 > \delta > 0$.  Then, there is a $c = c(p,q,d,\delta)$ such that for any finite family $\ff$ of at least $p$ convex sets in $\rr^d$ of diameter at least one each, if out of every $p$ sets in $\ff$, there are $q$ of them whose intersection has diameter at least one, then we can find $c$ convex sets $K_1, K_2, \ldots, K_c$ of diameter at least $1-\delta$ such that every set in $\ff$ contains at least one $K_i$.
\end{theorem}

In the volumetric version \cite[Thm. 1.2]{Soberon:2015tsa}, the lower bound on $q$ depends on heavily on $\delta$.  The proof of the original $(p,q)$ theorem is a tour de force of combinatorial geometry, and requires many classic results.  In order to prove Theorem \ref{theorem-p,q-diameter}, we give diameter version of these as well.  Notice that if one stubbornly refuses to check large families as Helly's theorem for the diameter requires, the result above gives non-trivial consequences with the same condition as B\'ar\'any, Katchalski and Pach used.

\begin{corollary}\label{corollary-chido}
	Let $d$ be a positive integer and $1 > \delta > 0$.  Let $\ff$ be a finite family of convex sets in $\rr^d$ such that the intersection of every $2d$ of them has diameter greater than or equal to one.  Then, $\ff$ may be split into $c(2d,2d,d,\delta)$ parts such that the diameter of the intersection of each of them is at least $1-\delta$.
\end{corollary}

  The loss of diameter $\delta$ is necessary in Theorem \ref{theorem-p,q-diameter} and Corollary \ref{corollary-chido}.  This is shown in section \ref{section-remarks} with a construction.
  
The rest of the paper is organised as follows.  In section \ref{section-width} we show Helly-type results for the property \textit{having $v$-width at least one}, for some fixed direction $v$.  In section \ref{section-colourful} we show a colourful version of Helly's theorem for the diameter.  In section \ref{section-fractional} we prove diameter versions of the fractional Helly theorem \cite{Katchalski:1979bq}, Tverberg's theorem \cite{Tverberg:1966tb}, B\'ar\'any's selection theorem (sometimes called the ``first selection lemma'') \cite{Barany:1982va} and the existence of weak $\varepsilon$-nets for convex sets \cite{Alon:2008ek} in order to prove Theorem \ref{theorem-p,q-diameter}.  Finally, in section \ref{section-remarks} we include some remarks and open problems.

\section{Results for fixed direction width}\label{section-width}

If instead of looking at the diameter, one is interested in the width  in a fixed direction $v$, we can obtain similar Helly-type to the ones mentioned in the introduction.  The original proofs for Helly-type theorems can be translated to this setting with minimal effort.  However, since some of them are useful for the results regarding the diameter, we present them completely here.

Let $v$ be a unit vector in $\rr^d$.  Given a compact convex set $K \subset \rr^d$, we say that $p \in K$ is a $v$-directional minimum if $\langle v, p \rangle \le \langle v, x\rangle$ for all $x \in K$, where $\langle \cdot, \cdot \rangle$ denotes the usual dot product.  We define a $v$-directional maximum similarly.  Throughout the rest of the paper we will assume that all convex sets we work with are compact and their boundary contains no segments.  This guarantees that the $v$-directional minimums for the sets and their non-empty finite intersections exist and are unique.  Standard approximation techniques show that there is no loss of generality.

Given a compact convex set $K$, we define its $v$-width as $\langle q, v\rangle - \langle p, v\rangle$ where $q, p$ are its $v$-directional maximum and $v$-direction minimum, respectively.

\begin{theorem}[Helly for $v$-width]\label{theorem-v-width-basic-helly}
	Let $v$ be a unit vector in $\rr^d$ and $\ff$ be a finite family of convex sets in $\rr^d$ such that the intersection of every $2d$ sets of $\ff$ has a $v$-width greater than or equal to one.  Then, the $v$-width of $\cap \ff$ is greater than or equal to one.
\end{theorem}

\begin{proof}
	Let $A$ be a subfamily of size $d$ whose $v$-directional minimum $p$ maximizes $\langle p, v\rangle$.  Given any other set $K _0 \in \ff$, let us show $p \in K_0$.  We know that $A \cup \{K_0\}$ must be intersecting, so let $u$ be a point of the intersection.  The minimality of $p$ implies $\langle u, v\rangle \ge \langle p, v\rangle$.
	
	If we denote by $K_1, K_2, \ldots, K_d$ the sets in $A$, every $d$-tuple of $A \cup \{K_0\}$ must be intersecting.  We call $p_i$ the $v$-directional minimum of $(A \cup \{K_0\})\setminus \{K_i\}$.  We know that $\langle p_i, v\rangle \le \langle p, v\rangle$ for each $i$ (also notice $p_0 = p$).  By convexity, there is a point $u_i \in (A \cup \{K_0\})\setminus \{K_i\}$ such that $\langle u_i, v \rangle = \langle p, v \rangle$ for each $i$.  This gives us $d+1$ points in the hyperplane $\{y : \langle y , v\rangle = \langle p, v \rangle \}$, of dimension $d-1$.
	
	By Radon's lemma \cite{Radon:1921vh}, these points can be partitioned into two sets $B, C$ such that $\conv (B) \cap \conv (C) \neq \emptyset$.  Let $p'$ be a point in $\conv (B) \cap \conv (C) \neq \emptyset$.  It is immediate that $p' \in K_i$ for all $0 \le i \le d$.  Thus, $p = p'$ and we have $p \in K_0$, as desired.
	
	Let $B$ be a subfamily of size $d$ with minimal $v$-directional maximum $q$.  Again, every set in the family contains $q$.  Since the $v$-width of $A \cup B$ is at least one, and this is realised by the segment $[p,q]$, the $v$-width of $\cap \ff$ is at least one. 
\end{proof}

\begin{theoremp}[Colourful Carath\'eodory for two points]\label{theorem-colourful-caratheodory}
	Given $2d$ sets of points $S_1, S_2, \ldots, S_{2d}$, and a set $S$ of two points $x,y$ such that $\{x,y\} \subset \conv (S_i)$ for each $1 \le i \le 2d$, there is a choice of points $s_1 \in S_1, \ldots, s_{2d} \in S_{2d}$ such that $\{x,y\} \in \conv \{s_1, \ldots, s_{2d}\}$
\end{theoremp}

This is a particular case of Theorem 1.3 in \cite{DeLoera:2015wp}.

\begin{theorem}[Colourful Helly for $v$-width]
	Let $\ff_1, \ff_2, \ldots, \ff_{2d}$ be finite families of convex sets in $\rr^d$, considered as colour classes.  If the $v$-width of the intersection of every rainbow choice $F_1 \in \ff_1, \ldots, F_{2d} \in \ff_{2d}$ is at least one, then there is a colour class $\ff_i$ such that the $v$-width of $\cap \ff_i$ is at least one.
\end{theorem}

\begin{proof}
A systematic way to obtain a colourful Helly theorem from a ``monochromatic'' version was presented in \cite[Thm. 5.3]{DeLoera:2015tc}.  Thus, it is sufficient to check the conditions of that result.  Let $\mathcal{P}(K)$ stand for \textit{``$K$ has $v$-width at least one''}.  Then, the following properties are satisfied.
\begin{itemize}
	\item $\mathcal{P}$ is a Helly property (Theorem \ref{theorem-v-width-basic-helly}), with Helly number $2d$.
	\item $\mathcal{P}$ is a monotone property.  i.e. if $K \subset K'$ then $\mathcal{P}(K)$ implies $\mathcal{P}(K')$.
	\item Let $v'$ be a unit vector in $\rr^d$ sufficiently close to $v$, but different.  We consider a $v'$-semispace a set of the form $\{x \in \rr^d : \langle x, v' \rangle \le \alpha \}$ for some $\alpha$.  Then, for every compact convex $K$ without segments in its boundary such that $\mathcal{P}(K)$ holds, there is a containment minimal $v'$-semispace $H$ such that $\mathcal{P}(K \cap H)$ holds.  Moreover, if we denote by $p, q$ the $v$-directional minimum and $v$-directional maximum of $K \cap H$ (which exist and are unique since $v' \neq v$), then every closed convex subset $K' \subset K \cap H$ with $\mathcal{P'}$ satisfies that $\conv\{p,q\} \subset K'$.
\end{itemize}
Having these properties, \cite[Thm. 5.3]{DeLoera:2015tc} implies the result we were seeking.
\end{proof}

The same idea leads to a fractional version.

\begin{theorem}\label{theorem-helly-v-width}
	Let $\alpha >0 $, $d$ a positive integer and $v$ a unit vector in $\rr^d$.  Then, there is a positive constant $\beta$ depending only on $\alpha, d$ such that for any family $\ff$ of $n$ convex sets in $\rr^d$ such that the intersection of at least $\alpha {{n}\choose{2d}}$ of the $2d$-tuples has $v$-width greater than or equal to one, there is a subfamily $\ff'$ of $\ff$ of cardinality at least $\beta n$ such that its intersection has $v$-width at least one.
\end{theorem}

\begin{proof}
Let $v' \neq v$ be a unit vector in $\rr^d$ sufficiently close to $v$.  For each subfamily $A$ of $\ff$ of cardinality $2d-1$ whose intersection has $v$-width at least one, let $H_A$ be the containment-minimal $v'$-semispace such that $\cap (A \cup \{H_A\})$ has $v$-width at least one.  Notice that if we consider $p_A, q_A$ the $v$-directional maximum and minimum of $\cap (A \cup \{H_A\})$ respectively (which exist and are unique since $v' \neq v$), then every convex set $C$ of $v$-width at least one such that $C \subset \cap (A \cup \{H_A\})$ satisfies $\{p_A, q_A\} \subset C$.

For each subfamily $B$ of $2d$ sets of $\ff$ whose intersection has $v$-width greater than or equal to one, let $A$ be its subfamily of size $2d-1$ with containment-maximal $H_A$.  Then, Theorem \ref{theorem-v-width-basic-helly} implies that the intersection of $B \cup \{H_A\}$ has $v$-width at least one, so all the sets in $B$ contain $\{p_A, q_A\}$

Consider the function that assigns to each $2d$-tuple $B$ with $v$-width at least one a $(2d-1)$-tuple $A$ as above.  Since a positive fraction of the $2d$-tuples satisfy this property, a direct double counting argument shows that there is a $(2d-1)$-tuple $A_0$ which was assigned at least $\beta n$ times for some positive $\beta$ depending only $d$ and $\alpha$.  Thus, at least $\beta n$ sets contain $\{p_{A_0}, q_{A_0}\}$, as desired.

\end{proof}

With the results above, the methods in Section \ref{section-fractional} can be carried out verbatim to the $v$-width case to prove the following result.

\begin{theorem}[$(p,q)$ theorem for $v$-width]
	Let $p \ge q \ge 2d$ be positive integers and $v$ a unit vector in $\rr^d$.  Then, there is a $c' = c'(p,q,d)$ such that for any finite family $\ff$ of at least $p$ convex sets of $v$-width at least one each, if out of every $p$ sets in $\ff$, there are $q$ of them whose intersection has $v$-width at least one, then we can find $c'$ convex sets $K_1, K_2, \ldots, K_{c'}$ of $v$-width at least one such that every set in $\ff$ contains at least one $K_i$.\end{theorem}

\section{Colourful Helly for the diameter}\label{section-colourful}

\begin{theorem}\label{theorem-colourful-diameter}
	There is an $n'=n'\left(\diam, d, \delta\right)$ such that for any $n'$ finite families  $\ff_1, \ff_2, \ldots, \ff_{n'}$ of convex sets in $\rr^d$, considered as colour classes, if the intersection of every colourful choice $F_1 \in \ff_1, \ldots, F_{n'} \in \ff_{n'}$ has diameter at least one, then there is a colour class $\ff_i$ with $\diam (\cap \ff_i) \ge 1- \delta$.  Moreover, $n'(\diam , d, \delta ) = \Theta_d \left(\delta^{-(d-1)/2}\right)$.
\end{theorem}

If $\ff_1 = \ldots = \ff_{n'}$, we obtain the monochromatic result, and the upper bound matches the one mentioned in the introduction.  The equivalent result for the volume \cite[Thm. 1.5]{Soberon:2015tsa} has a worse bound $n'(\vol, d, \delta) = O_d(\delta^{-(d^2-1)/4})$.

\begin{proof}
	Given a point $x \in S^{d-1}$, we denote by $C_{\delta}(x)$ the cap 
	\[
	C_{\delta}(x):= \left\{y \in S^{d-1}: \langle x, y \rangle \ge 1-{\delta} \right\}.
	\]
	We denote by $c_{\delta}$ its measure under the usual probability Haar measure of $S^{d-1}$.  It is known that $c_{\delta} = \Omega (\delta^{(d-1)/2})$ {\cite[Lemma 2.3]{ball1997elementary}}.
	
	Let $m= \left\lfloor\frac{1}{c_{\delta / 4}}\right\rfloor$ and consider $n' = 2dm$.  Assume that $\diam \ff_i < 1-\delta$ for all $i$.  We look for a contradiction.
	
	We can find $v_1, v_2, \ldots, v_{m}$ directions in $S^{d-1}$ such that for any $v \in S^{d-1}$, there is a $v_j$ with $\langle v, v_j \rangle \ge 1-\delta$.  In order to see this, take a set of points in $S$ of maximal cardinality such the caps $C_{\delta /4} (x)$ for $x \in S$ are have pairwise disjoint interiors.  By counting surface area one gets $|S| \le \lfloor{(c_{\delta / 4})^{-1}}\rfloor$.  However, if there was a direction $v$ not satisfying the conditions, an elementary geometric argument shows that we would be able to include $v$ in $S$, contradicting its maximality.
	
	For each $v_j$, consider $2d$ colour classes associated to it.  Since $\diam (\ff_i) < 1- \delta$ for all $i$, then their $v_j$-widths are also smaller than $1-\delta$.  Thus, by Theorem \ref{theorem-helly-v-width}, there must be a rainbow choice of these $2d$ colours such that the $v_j$-width of its intersection is strictly smaller than $1-\delta$.  Take $X$ to be the union of all these $2d$-tuples.  Notice that the $v_j$-width of $\cap X$ is strictly smaller than $1-\delta$ for all $v_j$.
	
	Let $\lambda = \diam (\cap X) \ge 1$, and $v$ a direction realising it.  Thus, $X$ contains a segment parallel to $v$ of length $\lambda$.  Let $v_j$ be such that $\langle v_j, v \rangle \ge 1- \delta$.  This implies that the $v_j$-width of $X$ is at least $1-\delta$, a contradiction.
\end{proof}

\section{Fractional and $(p,q)$ results for the diameter}\label{section-fractional}

In order to prove Theorem \ref{theorem-p,q-diameter}, we need to recreate the results needed for the proof of the original $(p,q)$ theorem for the diameter.  Simplified versions of Alon and Kleitman's method can be found in \cite{Alon:1996uf, Matousek:2002td}.  There are two main ingredients needed.  One is a fractional Helly theorem and the second is the existence for weak $\epsilon$-nets for convex sets of small size.  Their equivalents are Theorem \ref{theorem-fractional} and \ref{theorem-weak-nets} described below.

The structure of the proof we present here is the same.  However, some definitions, such as the one for weak $\varepsilon$-net, must be adapted.  In the case of volume, it is possible to recreate these results using properties of floating bodies \cite{Soberon:2015tsa}.  Namely, given a convex set $K$ of volume one, and $\varepsilon >0$, we define its floating body $K_{\varepsilon}$ as
\[
K_{\varepsilon} = K \setminus \cup \{H : H \ \mbox{is a halfsapce}, \vol (H \cap K) \le \varepsilon \}.
\]

There estimates on $\vol (K_{\varepsilon})$ allow for the proofs to work \cite{Barany:2010cy}.  For the diameter, there is no similar ``floating body''.  However, pigeonhole arguments on the directions realising the diameter, as in section \ref{section-colourful}, are sufficient.  Let us begin with a fractional Helly for diameter in the same spirit as \cite{Katchalski:1979bq}.

\begin{theorem}[Fractional Helly for the diameter]\label{theorem-fractional}
	Let $\alpha >0$, $1> \delta >0$ and $d$ a positive integer.  Then, there is a positive constant $\beta$ depending only on $\alpha, d, \delta$ such that for any finite family $\ff$ of $n$ convex sets in $\rr^d$ such that the intersection of at least $\alpha {{n}\choose{2d}}$ of the $2d$-tuples has diameter greater than or equal to one, there is a subfamily $\ff'$ of $\ff$ of cardinality at least $\beta n'$ such that its intersection has diameter at least $1-\delta$.	
\end{theorem}

The equivalent result for volume \cite[Thm. 1.4]{Soberon:2015tsa} has the disadvantage that the size of the subfamilies needed to check grows as $\delta$ decreases.  Namely, it needs to check families of size $O(\delta^{-(d^2-1)/4})$, which is much worse than the requirements of the volumetric Helly theorem.  Theorem \ref{theorem-fractional} is an example of a fractional Helly-type theorem which goes beyond its corresponding Helly theorem.  Such examples have appeared previously for set systems of bounded VC-dimension \cite{Matousek:2004cs}, for convexity on the integer lattice \cite{Anonymous:PHt9HPGF} or for checking the existence of hyperplane transversals in $\rr^d$ \cite{Alon:1995fs}.

\begin{proof}[Proof of Theorem \ref{theorem-fractional}]
	Consider the usual probability Haar measures on $S^{d-1}$.  For $y \in S^{d-1}$, let $C_\delta (y)$ be the set of points $x \in S^{d-1}$ such that $\langle x, y \rangle \ge 1-\delta$.  Let $c_{\delta}$ be the measure of $C_{\delta} (y)$.  A double counting argument shows that for any set $D$ of points in $S^{d-1}$, there must be a subset $D'$ of cardinality at least $c_{\delta} |D|$ and a point $v$ in $S^{d-1}$ such that $D' \subset C_{\delta} (v)$.
	
	For each such $2d$-tuple $B \subset \ff$, consider a direction $v_B$ realising its diameter.  Each of these directions can be represented by an antipodal pair on $S^{d-1}$.  Using the observation above, there must be a direction $v$ and set of at least $2 \alpha c_{\delta}{{n}\choose{2d}}$ of the $2d$-tuples of $\ff$ whose intersections have $v$-width greater than or equal to $1-\delta$.  Applying Theorem \ref{theorem-helly-v-width}, we are done.
\end{proof}

In order to get to the existence of weak $\varepsilon$-nets, we start by getting results showing that given a set of objects in $\rr^d$, there is a point $p$ that is ``sufficiently well surrounded'' by them.  The first result of this type is Tverberg's theorem.

Tverberg's theorem \cite{Tverberg:1966tb} says that given enough points in $\rr^d$, they can be split into $m$ parts such that the convex hulls of the parts intersect.  A point in this intersection is in some sense ``very deep'' within the original set of points.  In order to recreate this for the diameter, we need to work with sets with large diameter instead of points, giving the following statement.

\begin{theorem}[Tverberg's theorem for diameter]\label{theorem-tverberg-diameter}
	Let $d, m$ be positive integer, $1>\delta > 0$ and $n = \left\lfloor 4d^2(m-1)c_{\delta}^{-1}\right\rfloor+1$.  Given a family $\mathcal{T} = \{T_1, T_2, \ldots, T_{n}\}$ of sets in $\rr^d$ of diameter greater than or equal to one each, there is a partition of them into $m$ parts $A_1, A_2, \ldots, A_m$ so that
	\[
	\diam \left(\bigcap_{i=1}^m \conv(\cup A_i)\right) \ge 1- \delta.
	\] 
\end{theorem}

\begin{proof}
	By a double counting as before, there is a subfamily $\mathcal{T}' \subset \mathcal{T}$ of cardinality greater than or equal to $4d^2(m-1)+1$ and a direction $v$ such that the $v$-width of every member of $\mathcal{T}'$ is at least $1-\delta$.
	
	Now consider 
	\[
	\ff = \{\conv (\cup \G): \G \subset \mathcal{T}', |\G| = (m-1)(2d-1)2d+1 \}.
	\]
	Notice that every family forming an element of $\ff$ is missing at most $2d(m-1)$ members of $\mathcal{T}'$.  Thus, the intersection of every $2d$ of them contains some $T_i \in \mathcal{T'}$ which in turn implies that the $v$-width of their intersection is at least $1-\delta$.  Thus, by Theorem \ref{theorem-v-width-basic-helly} the $v$-width of $\cap \ff$ is at least $1-\delta$.
	
	Take two points $x, y \in \cap \ff$ that realise its $v$-width.  Every half-space containing either of them has non-empty intersection with at least $2d(m-1)+1$ sets of $\mathcal{T}'$.  Otherwise, it would contradict the fact that they are contained in $\cap \ff$.
	
	Thus, $x,y$ are contained in the convex hull of $\cup \mathcal{T}'$.  By the colourful Carath\'eodory theorem for two points (see Section \ref{section-width}) with $\ff_i = \cup \ttt'$ for $1 \le i \le 2d$, the set $\{x,y\}$ is contained in the convex hull of $2d$ points of $\cup \mathcal{T}'$.  If we remove the sets $T_i$ that generated these points and set them aside in a set called $A_1$, we have that every half-space containing either of $x,y$ has non-empty intersection with at least $2d(m-2)+1$ sets of what is left of $\mathcal{T}'$.
	Thus, we can continue this process and generate the desired sets $A_1, A_2, \ldots, A_m$ inductively. 
\end{proof}

Using Tverberg's theorem and colourful Carath\'edory, one can prove B\'ar\'any's selection theorem \cite{Barany:1982va}, also called the ``first selection lemma'' in \cite{Matousek:2002td}.  It says that, given a finite set $S$ of points in $\rr^d$, there is a point $p$ in a positive fraction of the simplices spanned by $S$.  This result holds in much more general settings, by either replacing the word ``simplex'' by the image of a different operator or requiring additional properties on the simplices containing the point \cite{Gromov:2010eb, Karasev:2012bj, Pach:1998vx, Magazinov:2015td}.  For our purposes we only need a diameter version of the original result by B\'ar\'any.

\begin{theorem}[Selection lemma for diameter]\label{theorem-selection-diameter}
	Let $d$ be a positive integer and $1 > \delta > 0$.  There is a constant $\lambda = \lambda (\delta, d)$ such that for any finite family $\ff$ of convex sets in $\rr^d$ of diameter one each, there is a set $K$ of diameter at least $1-\delta$ such that $K \subset \conv (\cup A)$ for at least $\lambda {{|\ff|}\choose{2d}}$ subsets $A \subset \ff$ of cardinality $2d$.  Moreover, $\lambda = \Omega_d( \delta ^{d(d-1)})$.
\end{theorem}

\begin{proof}
	First, there is a subfamily $\ff' \subset \ff$ of cardinality at least $c_{\delta} |\ff|$ and a direction $v$ such that the $v$-width of every member of $\ff'$ is at least $1-\delta$.  By the second part of the proof of Theorem \ref{theorem-tverberg-diameter}, there is a partition of $\ff'$ into at least $\frac{|\ff'|-1}{4d^2}+1$ parts such that the intersection of the convex hull of the union of the parts contains a set $K$ of $v$-width at least $1-\delta$.  Moreover, we may assume that $K$ is the convex hull of two points.
	 
	Colour each part by a different colour.  By the colourful Carath\'eodory theorem for two points (see section \ref{section-width}), for each $2d$-tuple of colours, there is an heterochromatic set whose convex hull contains $K$.  Thus, up to constants in the dimension there are at least
	\[
	{{\frac{|\ff'|-1}{4d^2}+1}\choose{2d}} \sim (4d)^{-2d}{{c_{\delta}|\ff|}\choose{2d}} \sim c_{\delta}^{2d}(4d)^{-2d}{{|\ff|}\choose{2d}}
	\]
	subsets $A \subset \ff$ of cardinality $2d$ such that $K \subset \conv(\cup A)$.
\end{proof}

The final ingredient needed is the existence of weak $\varepsilon$-nets for diameters.  The original results aims to find, for a given $S \subset \rr^d$, a set $T$ whose cardinality depends only on $\varepsilon$ and $d$ that intersects the convex hull of every subset of $S$ of cardinality at least $\varepsilon|S|$ \cite{Alon:2008ek}.  For our purposes we need both $S$ and $T$ to be families of sets with large diameter.

\begin{theorem}[Weak $\varepsilon$-nets for diameter]\label{theorem-weak-nets}
	Let $d$ be a positive integer, $1 > \delta > 0$, $1 > \varepsilon > 0$.  Then, there is a constant $m = m(d, \delta, \varepsilon)$ such that for any finite family $\ff$ of sets of diameter one each in $\rr^d$, there are $m$ sets $K_1, K_2, \ldots, K_m$ of diameter at least $1-\delta$ each such that for any subfamily $\ff' \subset \ff$ with $|\ff'| \ge \varepsilon|\ff|$, there is an $i$ satisfying $K_i \subset \conv (\cup \ff')$.  Moreover, $m(d, \delta, \varepsilon) = O_d (\varepsilon^{-2d}\cdot\delta^{-d(d-1)})$.
\end{theorem}

\begin{proof}
	We construct the set $\mathcal{K}=\{K_1, \ldots, K_m\}$ inductively, starting with an empty set.  Let $T$ be the number of $2d$-tuples $A \subset \ff$ such that $\conv (\cup A)$ contains no set in $\mathcal{K}$.  If there is a subfamily $\ff'$ with $|\ff'| \ge \varepsilon|\ff|$ such that $\conv (\cup \ff')$ does not contain any set of $\mathcal{K}$, we can apply Theorem \ref{theorem-selection-diameter} to $\ff'$.  Thus, we can find a set $K$ contained in the convex hull of the union of at least $\lambda {{|\ff'|}\choose{2d}} \sim \lambda \varepsilon^{2d} {{|\ff|}\choose{2d}}$ different subsets $A \subset \ff$ of cardinality $2d$, effectively reducing $T$ by that number if we add $K$ to $\mathcal{K}$.  The process cannot be repeated more than $O_d\left((\lambda \varepsilon^{2d})^{-1}\right)$ times, as desired.
\end{proof}

We call $\mathcal{K}$ a diameter weak $\varepsilon$-net for the pair $(\ff , \delta)$.  At this point we have all the ingredients needed to prove Theorem \ref{theorem-p,q-diameter}.  The proof of this theorem relies on the linear programming technique by Alon and Kleitman.  For this, we need the following definitions.

We consider $\mathcal{C}_{d, \delta} = \{F \subset \rr^d: \diam (F) \ge 1- \delta, \ F \ \mbox{is convex}\}$.  Then, given a finite family of convex sets $\ff$ in $\rr^d$, we define
\begin{itemize}
		\item the diameter transversal number $\tau_{\delta} (\ff)$ as the minimum $\sum_{C \in \C_{d,\delta}} w(C)$ over all functions $w:\C_{d,\delta} \to \{0,1\}$ such that
		\[
		\sum_{{C:C \subset F, \ C \in \C_{d,\delta}}} w(C) \ge 1
		\] for all $F \in \ff$,
		\item the fractional diameter transversal number $\tau^*_{\delta} (\ff)$ as the minimum $\sum_{C \in \C_{d,\delta}} w(C)$ over all functions $w:\C_{d,\delta} \to [0,1]$ such that
		\[
		\sum_{C: C \subset F, \ C \in \C_{d,\delta}} w(C) \ge 1
		\] for all $F \in \ff$, and
		\item the fractional diameter packing number $\nu^*_k (\ff)$ as the maximum $\sum_{F \in \ff} w(F)$ for $w: \ff \to [0,1]$ such that 
		\[
		\sum_{F: C \subset F, \ F \in \ff} w(F) \le 1
		\]
		for all $C \in \C_{d,\delta}$.
		\item Given two finite families $\ff$, $\ttt$ of convex sets in $\rr^d$, we say $\ttt$ is a $(1-\delta)$-diameter transversal for $\ff$ if every set in $\ttt$ has diameter at least $1- \delta $ and every set in $\ff$ contains at least one set in $\ttt$.  Note that if $w:\mathcal{C}_{d,\delta} \to \{0,1\}$ is a function satisfying the condition of $\tau_{\delta}(\ff)$, then the family $\{K\in \mathcal{C}_{d, \delta}: w(K)=1\}$ is a $(1-\delta)$-diameter transversal of $\ff$.
		\item We refer to the conditions of Theorem \ref{theorem-p,q-diameter} as the diameter $(p,q)$ condition.
	\end{itemize}

\begin{lemma}\label{lemma-one}
	Let $\ff$ be a finite family of convex sets in $\rr^d$, all with diameter at least one.  Then, $\tau_{\delta} (\ff)$ is bounded by a function depending only $\tau^*_{\delta/2}(\ff)$, $d$ and $\delta$.
\end{lemma}

\begin{proof}
	Consider a function $w:\mathcal{C}_{d, \delta/2} \to [0,1]$ which realises $\tau^{*}_{\delta/2}$.  Without loss of generality, we may assume that $w$ has finite support and only has rational values.  Let $M$ be the common denominator of $w(K)$ for all $K \in \mathcal{C}_{d, \delta/2}$.  Let $\ttt$ be the family that formed by the disjoint union of $M\cdot w(K)$ copies of $K$, for each $K \in \mathcal{C}_{d, \delta/2}$.  Now consider $\mathcal{K}$ a diameter weak $\left(\frac{1}{\tau^*_{\delta/2}(\ff)}\right)$-net of $(\ttt, \delta/2)$, as in Theorem \ref{theorem-weak-nets}.  Notice that the diameter of every member of $\mathcal{K}$ is at least $(1-\frac{\delta}{2})^2 \ge 1-\delta$.  By the definition of $\tau^*_{\delta/2}$, for $F \in \ff$, the number of copies of sets in $\ttt$ which are contained in $F$ is at least $(\tau^*_{\delta/2} (\ff))^{-1} |\ttt|$.  Thus, there is an element of $\mathcal{K}$ contained in $F$.  In other words, $\mathcal{K}$ is a $(1-\delta)$-diameter transversal to to $\ff$, so $\tau_{\delta} (\ff) \le |\mathcal{K}|$, which in turn is only bounded by a function of $\tau^*_{\delta/2} (\ff)$, $d$ and $\delta$.
\end{proof}

\begin{lemma}\label{lemma-two}
If $p \ge q \ge 2d$ and $\ff$ is a finite family of convex sets with the diameter $(p,q)$ condition, then $\nu^*_{\delta/2} (\ff)$ is bounded by a function that depends only on $p,q,d,\delta$.
\end{lemma}

\begin{proof}
	Let $w: \ff \to [0,1]$ be a function that realises $\nu^*_{\delta/2} (\ff)$.  We may assume without loss of generality that $w(C)$ is rational for all $C \in \ff$.  Let $w(C) = \frac{n_C}{m}$ where $m$ is the common denominator for all $w(C)$ for $C \in \ff$.  Let $\ff'$ be the family consisting of $n_C$ copies of $C$ for each $C \in \ff$ and let $N= |\ff'|$.  Note that $\frac{N}{m}= \sum_{C \in \ff} \frac{n_C}{m} = \nu^*_{\delta/2}(\ff)$.
	
	The family $\ff'$ satisfies the diameter $((q-1)(p-1)+1,q)$ property.  This comes immediately from the fact that every $[(q-1)(p-1)+1]$-tuple from $\ff'$ contains either $q$ copies of the same set or $p$ different sets of $\ff$.  In either case we have a $q$-tuple whose intersection has diameter at least one.  However, since $q\ge 2d$ this implies that there is a positive fraction of the $2d$-tuples of $\ff'$ whose intersection has diameter at least one.  Theorem \ref{theorem-fractional} implies then that there is a positive fraction $\beta$ depending only on $p,q,\delta$ such that there is a set $K_0$ of diameter at least $1-\frac{\delta}{2}$ contained in the intersection of at least $\beta N$ sets of $\ff'$.  Thus
	\[
	1 \ge \sum_{C \in \ff: K_0 \subset C}w(C) = \sum_{C \in \ff: K_0 \subset C} \frac{n_C}{m} \ge \frac{1}m \cdot \beta N = \beta \nu^*_{\delta/2}(\ff).
	\] 
	This implies $\nu^*_{\delta/2} (\ff) \le \frac{1}{\beta}$, as desired.
\end{proof}

\begin{proof}[Proof of Theorem \ref{theorem-p,q-diameter}]
	As in the Alon-Kleitman proof of the $(p,q)$ theorem, linear programming duality implies that $\nu_{\delta/2}^*(\ff) = \tau^*_{\delta/2}(\ff)$.  Thus, the lemmata \ref{lemma-one} and \ref{lemma-two} finish the proof.
\end{proof}

\section{Remarks and open problems}\label{section-remarks}

The Helly-type results we obtain for the diameter here improve upon their volumetric equivalent.  However, we see no reason they should not hold for the same strength in that setting. Corollary \ref{corollary-chido} seems like a good result to test for that purpose.

\begin{question}
Is there a constant $r(d, \delta)$ such that any finite family $\ff$ of convex sets in $\rr^d$ such that the intersection of every $2d$ of them has volume at least one can be partitioned into $r(d, \delta)$ parts so that the volume of the intersection of each part is at least $1-\delta$?	
\end{question}

Let us construct an example to show that the diameter loss $\delta$ is needed in Corollary \ref{corollary-chido}, Theorem \ref{theorem-p,q-diameter} and Theorem \ref{theorem-fractional}.

\begin{claim}
For any $k$, there is a family $\ff$ of $2dk+1$ convex sets in $\rr^d$ such that the intersection of any $2d$ of them has diameter at least one and for any partition of $\ff$ into $k$ parts, there is one whose intersection has diameter strictly smaller than one.
\end{claim}

\begin{proof}
	Let $n={{2k+1}\choose{2d}}$, and for each $2d$-tuple $A \subset \{1,2,\ldots, 2kd+1\}$, let $v_A$ be a pair of antipodal points in $\frac12S^{d-1}$.  For each $i \in \{1,2,\ldots, 2kd+1\}$, let
	\[
	K_i = \conv \left\{v_A: A \in {{[2kd+1]}\choose{2d}}, \ i \in A \right\}.
	\]
	
For any $2d$-tuple of sets, by construction their intersection contains some $v_A$, so the diameter is at least one.  Given a partition of $\{K_i : i\in [2kd+1]\}$ into $k$ parts, there must be one, say $P$, of cardinality at least $2d+1$.  For any $A \in {{[2kd+1]}\choose{2d}}$, there is an $K_i \in P$ for which $i \not\in A$, so $v_A \not\in \conv (\cap P)$.  Thus, $\cap P$ is contained in the interior of $\conv (\frac12S^{d-1})$ and is closed, so its diameter is strictly less than one.

If one wants shaper estimates, we can choose the antipodal points $v_A$ so that the circular caps $C_{\delta}  (v_A)$ are pairwise disjoint for some $\delta$, similarly to the argument of Theorem \ref{theorem-colourful-diameter}.
\end{proof}

The original conjecture by B\'ar\'any, Katchalski and Pach is still open, so we state it again to bring more attention to it.

\begin{conjecture}[B\'ar\'any, Katchalski, Pach \cite{Barany:1982ga}]
Let $\ff$ be a finite family of convex sets such that the intersection of every $2d$ of them has diameter at least one.  Then, $\diam (\cap \ff) \ge c d^{-1/2}$ for some absolute constant $c$.	
\end{conjecture}

For the conjecture above, the best lower bound on $\diam (\cap \ff)$ is $O(d^{-2d})$ \cite{Barany:1982ga}.

\bibliographystyle{amsalpha}

\bibliography{references.bib}

\noindent Pablo Sober\'on \\
\textsc{
Mathematics Department \\
Northeastern University \\
Boston, MA 02115, USA
}\\[0.1cm]

\noindent \textit{E-mail address: }\texttt{p.soberonbravo@neu.edu}

\end{document}